\documentclass[11pt,reqno]{amsart}

\usepackage[margin=0.9in]{geometry}
\usepackage{amsmath,amssymb,amsthm,mathtools}
\usepackage{enumitem}
\usepackage{needspace}
\usepackage[hidelinks]{hyperref}
\usepackage{microtype}

\newtheorem{theorem}{Theorem}[section]
\newtheorem{proposition}[theorem]{Proposition}
\newtheorem{lemma}[theorem]{Lemma}
\newtheorem{corollary}[theorem]{Corollary}
\theoremstyle{definition}

\theoremstyle{remark}
\newtheorem{remark}[theorem]{Remark}

\DeclareMathOperator{\rad}{rad}

\newcommand{\Pcal}{\mathcal P}

\title{Coprime matchings of divisor sets}
\author{Max A. Alekseyev}
\address{Department of Mathematics, The George Washington University,
Washington, DC 20052, USA}
\email{maxal@gwu.edu}
\author{Basile Beyer de Ryke}
\address{Brussels, Belgium}
\email{basile.beyer@gmail.com}
\date{}

\hypersetup{
 pdftitle={Coprime matchings of divisor sets},
 pdfauthor={Max A. Alekseyev and Basile Beyer de Ryke},
 pdfsubject={Coprime matchings and asymptotic enumeration}
}

\begin{document}

\begin{abstract}
For a positive integer $n$, let $D(n)$ be its set of positive divisors
and put $N:=\tau(n)=|D(n)|$.  We characterize the arithmetic
progressions $\Pcal$ of length $N$ that admit a bijection with $D(n)$
pairing relatively prime integers.  Such a coprime matching exists if
and only if, for every prime $p\mid n$, at most $N/(\nu_p(n)+1)$ terms
of $\Pcal$ are divisible by $p$, where $\nu_p(n)$ is the exponent of
$p$ in $n$.  Call $n$ an $M$-number if $p^p\nmid n$ for every prime
$p$.  We prove that $D(n)$ can be matched with every primitive
arithmetic progression of length $N$, meaning one whose first term
and common difference are coprime, if and only if $n$ is an
$M$-number.  This proves a conjecture of McNew and Pomerance.
In particular, $n$ is matchable, meaning that $D(n)$ can be matched
with the primitive progression $\{1,2,\ldots,N\}$, if and only if
$\lfloor N/p\rfloor\le N/(\nu_p(n)+1)$ for all $p\mid n$.
We classify the matchable non-$M$-numbers, settle a conjectured
divisor criterion, and prove that every divisor of a matchable
integer is matchable.  The proofs combine a Vandermonde determinant
with a partition into shorter progressions.  If $E(x)$ counts the
matchable non-$M$-numbers up to $x$, then, as $x\to\infty$,
\[
 E(x)\sim\frac{x}{\log x}
       \sum_{r\ge1}\frac{(\log\log x)^{r-1}}{(r-1)!\,q_r^{q_r}},
\]
where $q_r$ is the least prime greater than $2^r$.  This also gives
a second term in the asymptotic formula for the count of all
matchable integers up to $x$.
\end{abstract}

\maketitle

\section{Introduction}

Given two finite sets of integers with the same number of elements, a
\emph{coprime matching} between them is a bijection under which
corresponding integers are relatively prime.  We write
$[1,N]:=\{1,2,\ldots,N\}$.

We study coprime matchings from divisor sets to arithmetic progressions.
For a positive integer $n$, let
\[
 D(n):=\{d\in\mathbb Z_{>0}:d\mid n\},\qquad \tau(n):=|D(n)|.
\]
We write $\rad(n):=\prod_{p\mid n}p$, where the product ranges over primes,
and let $\nu_p(n)$ denote the exponent of $p$ in the prime factorization
of $n$.  We ask when $D(n)$ admits a coprime matching with a prescribed
arithmetic progression of length $\tau(n)$.  We call an arithmetic
progression \emph{primitive} if its first term and common difference
are coprime.  We take common differences to be positive; reversing
a progression handles a negative common difference.

For a prime $p\mid n$, exactly $\tau(n)/(\nu_p(n)+1)$ divisors of $n$
are not divisible by $p$.  Every term of the progression divisible by $p$ must be
matched to one of these divisors.  Our main result shows that the
resulting counting bounds, considered separately for each prime
dividing $n$, give a complete criterion.

\Needspace{9\baselineskip}
\begin{theorem}\label{thm:progression}
Let $n$ be a positive integer, put $N:=\tau(n)$, and let
\[
 \Pcal:=\{a+bt:t=0,1,\ldots,N-1\},\qquad
 a\in\mathbb Z,\quad b\in\mathbb Z_{>0}.
\]
For each prime $p\mid n$, put
\[
 c_p(\Pcal):=\#\{x\in\Pcal:p\mid x\}.
\]
Then $D(n)$ admits a coprime matching with $\Pcal$ if and only if
\begin{equation}\label{eq:progression-criterion}
 c_p(\Pcal)\le\frac{N}{\nu_p(n)+1}
 \qquad\text{for every prime }p\mid n.
\end{equation}
\end{theorem}

Necessity follows from the count above.  Sufficiency requires more:
terms divisible by different primes may need the same divisors as
coprime partners.  By Hall's theorem~\cite{Hall}, a coprime matching
exists if and only if, for every subset $X\subseteq\Pcal$, at least
$|X|$ divisors of $n$ are coprime to at least one element of $X$.
Theorem~\ref{thm:progression} shows that the counting bounds for
individual primes already guarantee all these inequalities, even
when the progression is not primitive.

The proof has two parts.  We first construct a matrix whose entries
are nonzero exactly when the corresponding integers are coprime.
An invertible change of columns gives a Vandermonde matrix formed
from distinct numbers.  The original matrix therefore has a nonzero
determinant, which proves that a coprime matching exists.  This construction applies
under explicit bounds on the exponents in the prime factorization
of $n$.  For the remaining cases allowed by
\eqref{eq:progression-criterion}, we partition the progression into
shorter progressions and combine matchings obtained from the same
construction.  The two special cases considered below use
$[1,\tau(n)]$ or require a matching with every primitive arithmetic
progression.

A classical related problem is a conjecture of D.~J. Newman: if $I$ is
any interval of $N$ consecutive integers, then $[1,N]$ admits a coprime
matching with $I$.  Pomerance and Selfridge proved this conjecture in
1980 \cite{PomeranceSelfridge}.  Robertson and Small later replaced
the interval by an arithmetic progression
\[
 S:=\{a+bt:t=0,1,\dots,N-1\},\qquad
 a\in\mathbb Z,\quad b\in\mathbb Z_{>0},
\]
and determined exactly when a coprime matching $[1,N]\to S$ exists
\cite{RobertsonSmall}.  In particular, such a matching exists whenever
$\gcd(a,b)=1$.  Thus $[1,N]$ admits a coprime matching with every primitive
arithmetic progression of length $N$.  More recent work has studied
the existence and number of coprime mappings and coprime permutations
\cite{PomeranceMatchings,PomerancePermutations,SahSawhney}; Bohman and
Peng also connected coprime mappings between intervals with the lonely
runner problem \cite{BohmanPeng}.

In 2022, Recam\'an called an integer $n$ \emph{matchable} if $D(n)$
admits a coprime matching with $[1,\tau(n)]$ and asked whether matchable
integers are more numerous than nonmatchable integers \cite{RecamanMO}.
For example, $54$ is matchable, whereas $36$ is not: its six even
divisors cannot all be paired with the five odd integers in
$[1,\tau(36)]=[1,9]$.
The ensuing discussion raised the problem of characterizing them.
Among the comments on that question,
K\"onig and MAA conjectured that every integer not divisible by $p^p$
for any prime $p$ should be matchable and that few integers outside
this class should be matchable; McNew and Pomerance record these
conjectures in \cite{McNewPomerance}.  K\"onig also asked
whether every divisor of a matchable integer must itself be matchable.
In the MathOverflow discussion \cite{RecamanMO}, MAA proposed the more precise criterion
\begin{equation}\label{eq:original-conjecture}
 \tau(n/d)\le
 \left\lceil\frac{(d-1)\tau(n)}{d}\right\rceil
 \qquad\text{for every }d\mid\rad(n),\ d>1.
\end{equation}
The condition is necessary, and the conjecture was verified
computationally in the MathOverflow discussion for $n\le 10^5$.

McNew and Pomerance subsequently initiated a systematic study of
matchable numbers \cite{McNewPomerance}.  They call $n$ an
\emph{$M$-number} if
\[
 p^p\nmid n\qquad\text{for every prime }p.
\]
They proved that almost every $M$-number is matchable, that every
squarefree integer (one with no squared prime factor) is matchable,
and that the integers belonging to exactly one of the classes of
matchable integers and $M$-numbers form a set of density zero.
Here the asymptotic density of a set is
the limiting proportion of integers in $[1,X]$ that belong to it as
$X\to\infty$; ``almost every'' allows a set of exceptions of density
zero.  Hence the matchable integers have asymptotic density
\[
 \alpha:=\prod_{p\ \mathrm{prime}}\left(1-\frac1{p^p}\right)
 =0.72199023441955\ldots.
\]
They explicitly endorsed the K\"onig--Alekseyev conjecture that
every $M$-number is matchable.  They also observed that this would
follow from K\"onig's question about divisors, because every
$M$-number divides one of the explicitly matchable numbers
$M_P:=\prod_{q\le P,\ q\ \mathrm{prime}}q^{q-1}$; they did not find a
proof that matchability passes to divisors.

McNew and Pomerance also considered a stronger requirement.  They use
``coprime arithmetic progression'' for what we call a primitive
progression \cite{McNewPomerance}.  A positive integer
$n$ is \emph{strongly matchable} if $D(n)$ admits a coprime matching
with every primitive arithmetic progression of length $\tau(n)$.
The Robertson--Small theorem gives such a matching from
$[1,\tau(n)]$; the question is whether $D(n)$ can replace this interval.  McNew and
Pomerance conjectured that $n$ is strongly matchable if and only if it
is an $M$-number \cite[Conjecture~1]{McNewPomerance}.

We obtain the following characterization, proving this conjecture.

\begin{theorem}\label{thm:strong}
A positive integer is strongly matchable if and only if it is an
$M$-number.
\end{theorem}

The forward implication is \cite[Proposition~4]{McNewPomerance}.
The Vandermonde construction gives the converse directly, before we
prove the full progression criterion.

For matchings with $[1,\tau(n)]$, we obtain the following equivalent
arithmetic criteria, including \eqref{eq:original-conjecture}.

\begin{theorem}\label{thm:characterization}
Let $n$ be a positive integer, put $N:=\tau(n)$, and let $\nu_p(n)$
denote the exponent of $p$ in the prime factorization of $n$.
The following are equivalent.
\begin{enumerate}[label=\textup{(\roman*)}]
 \item $n$ is matchable.
 \item For every prime $p\mid n$,
 \begin{equation}\label{eq:prime-criterion-intro}
   \left\lfloor\frac{N}{p}\right\rfloor
   \le \frac{N}{\nu_p(n)+1}.
 \end{equation}
 \item For every $d\mid\rad(n)$, $d>1$,
 \begin{equation}\label{eq:d-criterion-intro}
   \tau(n/d)\le
   \left\lceil\frac{(d-1)N}{d}\right\rceil.
 \end{equation}
 \item Either $n$ is an $M$-number, or exactly one prime $p$
 satisfies $p^p\mid n$, and for this prime
 \begin{equation}\label{eq:structural-intro}
   \bigl(\nu_p(n)+1-p\bigr)
   \tau\!\left(\frac{n}{p^{\nu_p(n)}}\right)<p.
 \end{equation}
\end{enumerate}
\end{theorem}

For $\Pcal=[1,N]$, we have $c_p(\Pcal)=\lfloor N/p\rfloor$, so
\eqref{eq:prime-criterion-intro} is precisely the counting condition
in Theorem~\ref{thm:progression}.  When $p^p\mid n$, McNew and Pomerance
used the bound $N/(\nu_p(n)+1)\le N/(p+1)$ to obtain a weaker necessary
condition \cite[Proposition~1]{McNewPomerance}.
The equivalence with part~(iv) explicitly describes the matchable
non-$M$-numbers.  It also answers K\"onig's question about divisors of
matchable integers.

\begin{corollary}\label{cor:divisor-closed-intro}
Every divisor of a matchable integer is matchable.
\end{corollary}

The structural classification also permits a more precise count.
For real $x\ge1$, put
\[
 E(x):=\#\{n\le x:n\text{ is matchable but is not an }M\text{-number}\}.
\]
McNew and Pomerance~\cite{McNewPomerance} proved
$E(x)\le x/(\log x)^{1+o(1)}$ and noted the lower bound
$E(x)\gg x/\log x$ supplied by integers of the form $27q$, with $q$
prime.  We determine an asymptotic formula for $E(x)$.
All logarithms are natural; write $\log_1x:=\log x$ and
$\log_{j+1}x:=\log(\log_jx)$ whenever these expressions are defined.
For each integer $r\ge1$, let $q_r$ be the least prime greater than
$2^r$, and define
\begin{equation}\label{eq:Psi-definition}
 \Psi(t):=\sum_{r\ge1}\frac{t^{r-1}}{(r-1)!\,q_r^{q_r}}.
\end{equation}
This series converges for every complex $t$.

\begin{theorem}\label{thm:exceptional-count}
As $x\to\infty$,
\[
 E(x)\sim\frac{x}{\log x}\Psi(\log_2x).
\]
\end{theorem}

The first terms of the function in \eqref{eq:Psi-definition} are
\[
 \Psi(t)=\frac1{3^3}+\frac{t}{5^5}
          +\frac{t^2}{2!\,11^{11}}+\frac{t^3}{3!\,17^{17}}+\cdots.
\]
They correspond to $3^3$ times a prime, $5^5$ times a product of two
distinct primes, $11^{11}$ times a product of three distinct primes,
and so on, with the displayed prime excluded from the squarefree
part.  The transitions occur slowly: even the second term of
$\Psi(\log_2x)$ exceeds the first only when $\log_2x>5^5/3^3$.
Theorem~\ref{thm:Psi-growth} quantifies the growth of
$\Psi$.  In particular, $E(x)\log x/x$ tends to infinity faster than
any fixed power of $\log_2x$.  If
\[
 A(x):=\#\{n\le x:n\text{ is matchable}\},
\]
then Corollary~\ref{cor:matchable-second-term} gives a second term
after $\alpha x$ in the asymptotic formula for $A(x)$.

Section~\ref{sec:vandermonde} gives the determinant construction, and
Section~\ref{sec:strong} deduces the characterization of strongly
matchable integers.  Section~\ref{sec:progressions} proves the full
progression criterion and characterizes the common differences for
which every primitive progression admits a matching.
Section~\ref{sec:ordinary} derives the arithmetic criteria for
matchability, describes the matchable non-$M$-numbers, and proves
that matchability passes to divisors.
Sections~\ref{sec:exceptional-count} and~\ref{sec:dominant-families} establish the asymptotic count and
describe the families that supply its main term.

\section{A coprime matching from a Vandermonde determinant}
\label{sec:vandermonde}

We begin with the algebraic construction.  The exponent bounds are
needed only at primes not dividing the common difference, a fact we
will use in Section~\ref{sec:progressions}.

\begin{lemma}\label{lem:vandermonde}
Let
\[
 m:=\prod_{i=1}^r p_i^{h_i-1},
\]
where $p_1,\ldots,p_r$ are distinct primes and the $h_i$ are positive
integers, and put
\[
 Q:=\tau(m)=\prod_{i=1}^r h_i.
\]
Let
\[
 \Pcal:=\{a+bt:t=0,1,\dots,Q-1\},\qquad
 a\in\mathbb Z,\quad b\in\mathbb Z_{>0},\quad \gcd(a,b)=1.
\]
If $h_i\le p_i$ whenever $p_i\nmid b$, then $D(m)$ admits a coprime
matching with $\Pcal$.
\end{lemma}

\begin{proof}
If $r=0$, then $m=Q=1$, and the only divisor, $1$, is coprime
to the only term of the progression.  Assume $r\ge1$ and fix
a progression $\Pcal$ as in the statement.
Write $[r]:=\{1,2,\dots,r\}$ and split the indices into
\[
 I_0:=\{i\in[r]:p_i\nmid b\},\qquad
 I_1:=\{i\in[r]:p_i\mid b\}.
\]
For each $i\in I_0$, choose a complex number $\zeta_i$ with
$\zeta_i^{p_i}=1$ and $\zeta_i\ne1$.  Since $p_i$ is prime,
$\zeta_i^k=1$ for an integer $k$ exactly when $p_i\mid k$.
For $i\in[r]$, define
\[
 y_i(t):=
 \begin{cases}
  \zeta_i^{a+bt}-1,&i\in I_0,\\
  p_i^t,&i\in I_1.
 \end{cases}
\]
We list the divisors of $m$ in increasing order whenever they index
a matrix.  Consider the $Q\times Q$ matrix $M$, with rows indexed by
$t=0,1,\dots,Q-1$ and columns indexed by the divisors $d\mid m$,
whose entries are
\begin{equation}\label{eq:M-entry}
 M_{t,d}:=\prod_{i=1}^r y_i(t)^{\nu_{p_i}(d)}.
\end{equation}
We use the convention $0^0=1$ in \eqref{eq:M-entry}.

For $i\in I_0$, the choice of $\zeta_i$ gives
\[
 y_i(t)=0
 \quad\Longleftrightarrow\quad
 p_i\mid a+bt.
\]
For $i\in I_1$, the relations $p_i\mid b$ and $\gcd(a,b)=1$ imply
$p_i\nmid a$, hence $p_i\nmid a+bt$ for every $t$.  Consequently
\begin{equation}\label{eq:support}
 M_{t,d}\ne0
 \quad\Longleftrightarrow\quad
 \gcd(a+bt,d)=1.
\end{equation}
Thus a nonzero entry in row $t$ and column $d$ corresponds
exactly to a coprime pair $(a+bt,d)$.

It remains to prove that $\det M\ne0$.  For $i\in I_0$ and
$0\le e<h_i$,
\begin{align}
 y_i(t)^e
 &=\bigl(\zeta_i^{a+bt}-1\bigr)^e \notag\\
 &=\sum_{c=0}^e(-1)^{e-c}\binom ec
   \zeta_i^{ac}\bigl(\zeta_i^{bc}\bigr)^t.
 \label{eq:triangular-change}
\end{align}
For $i\in I_1$ we simply have
\[
 y_i(t)^e=(p_i^e)^t.
\]

For each divisor $f\mid m$, write uniquely $f=f_0f_1$, where
$\gcd(f_0,b)=1$ and $\rad(f_1)\mid\rad(b)$.  Thus $f_1$ is precisely
the product of the prime-power factors of $f$ whose primes divide
$b$.  Set
\begin{equation}\label{eq:nodes}
 \lambda_f:=
 f_1\prod_{i\in I_0}\zeta_i^{b\nu_{p_i}(f)},
\end{equation}
and let $V$ be the matrix
\[
 V_{t,f}:=\lambda_f^{\,t}.
\]
Expanding \eqref{eq:M-entry} with \eqref{eq:triangular-change} gives
\begin{equation}\label{eq:MVA}
 M=VA,
\end{equation}
where $A$ is indexed by divisors $f,d\mid m$.  Its entry $A_{f,d}$
is zero unless
\[
 \nu_{p_i}(f)\le \nu_{p_i}(d)\quad(i\in I_0),
 \qquad
 \nu_{p_i}(f)=\nu_{p_i}(d)\quad(i\in I_1),
\]
and otherwise
\[
 A_{f,d}
 =\prod_{i\in I_0}
 (-1)^{\nu_{p_i}(d)-\nu_{p_i}(f)}
 \binom{\nu_{p_i}(d)}{\nu_{p_i}(f)}
 \zeta_i^{a \nu_{p_i}(f)}.
\]
Each coefficient $A_{f,d}$ is a product of coefficients contributed
by the individual primes.  For $i\in I_0$, these coefficients come
from \eqref{eq:triangular-change}; for $i\in I_1$, the exponent is
unchanged.  In particular, $A_{f,d}\ne0$ implies $f\mid d$.
With our ordering of the divisors, $A$ is therefore upper triangular,
and its diagonal entries are
\[
 A_{d,d}=\prod_{i\in I_0}\zeta_i^{a \nu_{p_i}(d)}\ne0.
\]
Hence $\det A\ne0$.

We next verify that the $Q$ numbers $\lambda_f$ in \eqref{eq:nodes} are distinct.
Suppose $\lambda_f=\lambda_g$ for two divisors $f,g\mid m$, and
write $f=f_0f_1$ and $g=g_0g_1$ as above.  Then
\begin{equation}\label{eq:node-equality}
 \frac{f_1}{g_1}
 =
 \prod_{i\in I_0}\zeta_i^{-b(\nu_{p_i}(f)-\nu_{p_i}(g))}.
\end{equation}
The left side is a positive rational number, while the right side has
absolute value $1$.  Hence both sides of \eqref{eq:node-equality} equal
$1$, and therefore $f_1=g_1$.
For $i\in I_0$, put
\[
 \Delta_i:=\nu_{p_i}(f)-\nu_{p_i}(g).
\]
Then
\[
 \prod_{i\in I_0}\zeta_i^{b\Delta_i}=1.
\]
Fix $j\in I_0$ and raise this identity to
\[
 P_j:=\prod_{\substack{i\in I_0\\ i\ne j}}p_i.
\]
All factors with $i\ne j$ disappear, and therefore
\[
 \zeta_j^{bP_j\Delta_j}=1.
\]
Because $p_j$ is prime, $p_j\nmid b$, and $p_j\nmid P_j$, we obtain
$p_j\mid\Delta_j$.  On the other hand,
\[
 |\Delta_j|\le h_j-1\le p_j-1,
\]
so $\Delta_j=0$.  Thus $\nu_{p_j}(f)=\nu_{p_j}(g)$ for every $j\in I_0$.
Hence $f_0=g_0$; together with $f_1=g_1$, this gives $f=g$.
Thus the numbers in \eqref{eq:nodes} are pairwise distinct.

The matrix $V$ is a Vandermonde matrix, with determinant
\[
 \det V=\prod_{\substack{f,g\mid m\\f<g}}(\lambda_g-\lambda_f)\ne0.
\]
Equation~\eqref{eq:MVA} now gives $\det M=\det V\det A\ne0$.
Finally,
\[
 \det M=\sum_{\sigma}\operatorname{sgn}(\sigma)
       \prod_{t=0}^{Q-1}M_{t,\sigma(t)},
\]
where $\sigma$ ranges over bijections from the row indices to the
divisors of $m$.  Since this sum is nonzero, at least one $\sigma$ has
\(\prod_t M_{t,\sigma(t)}\ne0\).  By \eqref{eq:support},
$\gcd(a+bt,\sigma(t))=1$ for every $t$.  The bijection therefore
gives a coprime matching between $\Pcal$ and $D(m)$.
\end{proof}

\begin{remark}\label{rem:det-certificate}
The determinant expansion works over any field: a square matrix
with nonzero determinant has a choice of nonzero entries containing
exactly one entry from each row and each column.  In particular,
suppose that its rows and columns are indexed by the two equally
large vertex sets of a bipartite graph, and that an entry is zero
whenever its corresponding vertices are not joined by an edge.
A nonzero determinant then gives a matching covering every vertex.  The construction above makes
the nonzero entries correspond to coprime pairs.
The bound $h_i\le p_i$ is used only to turn a congruence modulo $p_i$
into equality.  If $h_i>p_i$ for a prime $p_i\nmid b$, the relevant
powers of $\zeta_i$ repeat.  Primes dividing $b$ are distinguished by
the positive factors in \eqref{eq:nodes} and require no exponent
bound.
\end{remark}

\section{Strongly matchable numbers}
\label{sec:strong}

We now prove Theorem~\ref{thm:strong}.

\begin{proof}[Proof of Theorem~\ref{thm:strong}]
First suppose that $n$ is an $M$-number.  Write
\[
 n=\prod_{i=1}^r p_i^{\nu_{p_i}(n)},\qquad
 h_i:=\nu_{p_i}(n)+1.
\]
The $M$-number condition is precisely $h_i\le p_i$ for every $i$.
For every positive common difference, the hypotheses of
Lemma~\ref{lem:vandermonde} therefore hold with $m=n$.  The lemma gives
a coprime matching of $D(n)$ with every primitive arithmetic
progression of length $\tau(n)$.  Hence $n$ is strongly matchable.

For completeness, we include the short argument for the converse
implication already proved in \cite[Proposition~4]{McNewPomerance}.
Suppose that $n$ is not an $M$-number.  Then for some prime $p$,
\[
 \nu_p(n)\ge p.
\]
Put $N:=\tau(n)$ and consider an interval of $N$ consecutive positive
integers beginning at a multiple of $p$.  It contains
$\lceil N/p\rceil$ multiples of $p$.  Each would have to be matched
to a divisor of $n$ not divisible by $p$, but there are only
$N/(\nu_p(n)+1)$ such divisors.  Since
\[
 \frac{N}{\nu_p(n)+1}<\frac Np
 \le\left\lceil\frac Np\right\rceil,
\]
no coprime matching exists.
This interval is a primitive arithmetic progression, and therefore
$n$ is not strongly matchable.
\end{proof}

\Needspace{7\baselineskip}
The theorem also gives the following density result, strengthening
\cite[Proposition~6]{McNewPomerance}.

\begin{corollary}\label{cor:strong-density}
The strongly matchable integers have asymptotic density
\[
 \alpha=\prod_p\left(1-\frac1{p^p}\right)
 =0.72199023441955\ldots.
\]
\end{corollary}

\begin{proof}
By Theorem~\ref{thm:strong}, the strongly matchable integers are
exactly the integers not divisible by $p^p$ for any prime $p$.
For any fixed $P$, the Chinese remainder theorem shows that the
integers avoiding $p^p$ for primes $p\le P$ have density
$\prod_{p\le P}(1-p^{-p})$.  For every positive integer $X$, the
proportion of integers in $[1,X]$ divisible by $p^p$ for at least one
prime $p>P$ is at most $\sum_{p>P}p^{-p}$.  This bound tends to zero
as $P\to\infty$, because $\sum_p p^{-p}$ converges.  Taking
$P\to\infty$ gives the claimed product.
\end{proof}

\section{Matchings with prescribed arithmetic progressions}
\label{sec:progressions}

We first isolate an arithmetic consequence of the necessary
counting inequalities.  The lemma will be used here for primes not
dividing the common difference and again in the characterization
of matchable integers.

\begin{lemma}\label{lem:one-supercritical}
Let $n$ be a positive integer, put $N:=\tau(n)$, and let $p\mid n$
be a prime with $\nu_p(n)\ge p$.  Set
\[
 c:=\nu_p(n)+1-p,\qquad T:=\tau(n/p^{\nu_p(n)}).
\]
Then
\begin{equation}\label{eq:exceptional-prime-equivalence}
 \left\lfloor\frac Np\right\rfloor
       \le\frac{N}{\nu_p(n)+1}
 \quad\Longleftrightarrow\quad cT<p.
\end{equation}
At most one prime divisor of $n$ can satisfy both
$\nu_p(n)\ge p$ and the inequality on the left.
\end{lemma}

\begin{proof}
Since $N=(p+c)T$,
\[
 \left\lfloor\frac Np\right\rfloor
   =T+\left\lfloor\frac{cT}{p}\right\rfloor,
\]
which proves the equivalence.
If $p<q$ were two primes satisfying the stated conditions, then
$T=\tau(n/p^{\nu_p(n)})$ would be divisible by $\nu_q(n)+1$.
Consequently $T\ge\nu_q(n)+1>q>p$, contradicting $cT<p$ and $c\ge1$.
\end{proof}

We now construct the matching under the counting inequalities.
The exponent bounds in Lemma~\ref{lem:vandermonde} are required only
for primes not dividing the common difference.

\begin{proof}[Proof of Theorem~\ref{thm:progression}]
Necessity follows because every term of the progression divisible by
$p$ must be matched to one of the $N/(\nu_p(n)+1)$ divisors not divisible by $p$.

For sufficiency, assume \eqref{eq:progression-criterion} and put
$g:=\gcd(a,b)$.  If a prime $p\mid n$ divided $g$, then
$c_p(\Pcal)=N$, contradicting $N/(\nu_p(n)+1)<N$.
Hence $\gcd(g,n)=1$.  Dividing every term by $g$ does not change
which divisors of $n$ are coprime to it, or any of the counts $c_p$,
and gives a primitive progression.  We may therefore assume
$\gcd(a,b)=1$.

For every prime $p\mid n$ with $p\nmid b$,
\[
 \left\lfloor\frac Np\right\rfloor
 \le c_p(\Pcal)\le\frac{N}{\nu_p(n)+1}.
\]
Lemma~\ref{lem:one-supercritical} therefore shows that at most one
prime $p$ satisfies both $p^p\mid n$ and $p\nmid b$.

If no such prime exists, Lemma~\ref{lem:vandermonde} applies directly
with $m=n$.  Otherwise, let $p$ be the unique prime satisfying these
two conditions, and put
\[
 m:=\frac{n}{p^{\nu_p(n)}},\qquad
 T:=\tau(m),\qquad
 c:=\nu_p(n)+1-p.
\]
We have
\[
 c T<p,\qquad N=(p+c)T,
\]
and every prime $q\mid m$ with $q\nmid b$ satisfies
$\nu_q(m)+1\le q$.  No exponent bound is required at primes dividing
$b$.

The first $pT$ terms of $\Pcal$ split into the progressions
\[
 \mathcal B_r:=\{a+br+pbk:k=0,1,\ldots,T-1\},
 \qquad 0\le r<p.
\]
Since $p\nmid b$, the integers $a,a+b,\ldots,a+(p-1)b$ have
distinct remainders upon division by $p$.  Thus exactly one index
$r_0$ satisfies $p\mid a+br_0$.  The progression $\mathcal B_{r_0}$
contains $T$ multiples of $p$, and the other progressions contain none.  The assumed bound $c_p(\Pcal)\le T$
therefore implies that none of the final $c T$ terms is divisible
by $p$.

The progression
\[
 p^{-1}\mathcal B_{r_0}
 =\left\{\frac{a+br_0}{p}+bk:k=0,1,\ldots,T-1\right\}
\]
is primitive: any common divisor of its first term and $b$ would
divide both $a$ and $b$.  By Lemma~\ref{lem:vandermonde}, choose
a coprime matching from $D(m)$ to this progression.  Multiplying each
matched term by $p$ gives a coprime matching from $D(m)$ to
$\mathcal B_{r_0}$, since no divisor of $m$ is divisible by $p$.

For $r\ne r_0$, the progression $\mathcal B_r$ is primitive with
common difference $pb$: its first term is coprime to $b$ and is not
divisible by $p$.  Split the final $c T$ terms into the following
$c$ progressions:
\[
 \mathcal C_s:=
 \{a+b(pT+(s-1)T+k):k=0,1,\ldots,T-1\},
 \qquad 1\le s\le c.
\]
Each $\mathcal C_s$ is primitive with common difference $b$.
Lemma~\ref{lem:vandermonde} applies to $D(m)$ and each of these
$p-1+c=\nu_p(n)$ remaining progressions.  Indeed, since $p\nmid m$,
the required exponent bounds for primes dividing $m$ are the same
for common differences $b$ and $pb$.

Pair these progressions, in any order, with the sets
\[
 pD(m),p^2D(m),\ldots,p^{\nu_p(n)}D(m),
\]
where $p^eD(m):=\{p^e d:d\mid m\}$.
For the progression paired with $p^eD(m)$, replace each matched
divisor $d$ by $p^e d$.  Since none of the terms of these progressions
is divisible by $p$, every resulting pair remains coprime.  Together
with the matching from $D(m)$ to $\mathcal B_{r_0}$, these matchings
form the required bijection.
\end{proof}

We next fix the common difference and ask when every primitive
progression of length $\tau(n)$ with that difference can be matched
with $D(n)$.

\begin{corollary}\label{cor:fixed-step}
Let $n$ be a positive integer and fix $b\in\mathbb Z_{>0}$.
Then $D(n)$ admits a coprime matching with every primitive arithmetic
progression of length $\tau(n)$ and common difference $b$ if and only if
every prime $p$ for which $p^p\mid n$ divides $b$.
\end{corollary}

\begin{proof}
If every prime $p$ satisfying $p^p\mid n$ divides $b$,
Lemma~\ref{lem:vandermonde} applies to every such progression.
Conversely, suppose $p^p\mid n$ and $p\nmid b$, and put $N:=\tau(n)$.
The primitive progression
\[
 \{p+bt:t=0,1,\ldots,N-1\}
\]
contains $\lceil N/p\rceil$ multiples of $p$.  Since
\[
 \frac{N}{\nu_p(n)+1}<\frac Np
 \le\left\lceil\frac Np\right\rceil,
\]
it violates the necessary condition in
Theorem~\ref{thm:progression}.
\end{proof}

\section{Characterization and consequences for matchable integers}
\label{sec:ordinary}
\label{sec:consequences}

For the initial interval $[1,N]$, the number of multiples of a prime
$p$ is $\lfloor N/p\rfloor$.  Thus
Theorem~\ref{thm:progression} immediately gives the equivalence
between matchability and
\eqref{eq:prime-criterion-intro}.
There is also a direct counting reason for the condition originally
conjectured in \cite{RecamanMO}.  For any $d>1$ dividing $n$, exactly
$\tau(n/d)$ divisors of $n$ are divisible by $d$.  In a coprime
matching they must be paired with integers not divisible by $d$,
of which there are
$N-\lfloor N/d\rfloor=\lceil(d-1)N/d\rceil$ in $[1,N]$.
The next lemma shows that the inequalities for prime $d$ already
imply all the others.

\Needspace{8\baselineskip}
\begin{lemma}\label{lem:equivalent-d}
For a positive integer $n$ with $N:=\tau(n)$, the following are
equivalent:
\begin{enumerate}[label=\textup{(\alph*)}]
 \item \eqref{eq:prime-criterion-intro} holds for every prime $p\mid n$;
 \item
 \begin{equation}\label{eq:d-criterion}
  \tau(n/d)\le
  \left\lceil\frac{(d-1)N}{d}\right\rceil
  \qquad (d\mid\rad(n),\ d>1).
 \end{equation}
\end{enumerate}
\end{lemma}

\begin{proof}
For a prime $p\mid n$,
\[
 \tau(n/p)=N-\frac{N}{\nu_p(n)+1},
\]
while, because $N$ is an integer,
\[
 \left\lceil\frac{(p-1)N}{p}\right\rceil
 =N-\left\lfloor\frac{N}{p}\right\rfloor.
\]
Hence the $d=p$ instance of \eqref{eq:d-criterion} is exactly
\eqref{eq:prime-criterion-intro}.  Thus (b) implies (a).

Conversely, suppose (a), let $d\mid\rad(n)$ with $d>1$, and choose a prime
$p\mid d$.  Since $n/d$ divides $n/p$,
\[
 \tau(n/d)\le\tau(n/p).
\]
Using the prime case and $p\le d$ gives
\begin{align*}
 \tau(n/d)
 &\le\tau(n/p)\\
 &\le\left\lceil\frac{(p-1)N}{p}\right\rceil\\
 &\le\left\lceil\frac{(d-1)N}{d}\right\rceil.\qedhere
\end{align*}
\end{proof}

The same argument applies to every divisor $d>1$ of $n$, without
requiring $d$ to be squarefree.  Thus the equivalent criterion may
also be imposed for all $d\mid n$, $d>1$; the squarefree formulation
in Theorem~\ref{thm:characterization} retains the original conjecture.

\begin{proof}[Proof of Theorem~\ref{thm:characterization}]
Theorem~\ref{thm:progression}, applied to $[1,N]$, gives
(i)$\Leftrightarrow$(ii).  Lemma~\ref{lem:equivalent-d} gives
(ii)$\Leftrightarrow$(iii).

Suppose (ii).  If $n$ is not an $M$-number,
Lemma~\ref{lem:one-supercritical} gives a unique prime $p$ with
$p^p\mid n$, and \eqref{eq:exceptional-prime-equivalence} gives
\eqref{eq:structural-intro}.  Hence (iv) holds.
Conversely, if $n$ is an $M$-number, then
\[
 \left\lfloor\frac Np\right\rfloor
       \le\frac Np\le\frac{N}{\nu_p(n)+1}
 \qquad(p\mid n).
\]
In the other case of (iv), the same inequality holds at every prime
other than the unique $p$ with $p^p\mid n$.  At that prime, the
inequality is equivalent to \eqref{eq:structural-intro} by
Lemma~\ref{lem:one-supercritical}.  This proves
(iv)$\Rightarrow$(ii).
\end{proof}

For matchings with $[1,N]$, the partition can be written explicitly.
For a matchable non-$M$-number, write
\[
 n=p^{p+c-1}m,\qquad T:=\tau(m),\qquad cT<p,
\]
where $p$ is prime, $c$ is a positive integer, $m$ is an $M$-number,
and $p\nmid m$.  Partition $[1,pT]$ into
\[
 C_0:=\{p,2p,\ldots,Tp\},\qquad
 C_r:=\{r,r+p,\ldots,r+(T-1)p\}\quad(1\le r<p),
\]
and partition the remaining interval $[pT+1,(p+c)T]$ into
\[
 B_j:=\{pT+(j-1)T+1,\ldots,pT+jT\}\quad(1\le j\le c).
\]
Match $D(m)$ with $C_0$ by taking a coprime matching with $[1,T]$
and multiplying each matched term by $p$.
The other $p-1+c$ progressions contain no multiples of $p$ and are
matched with the sets $p^eD(m)$, $1\le e\le p+c-1$, exactly as in
the proof of Theorem~\ref{thm:progression}.

\Needspace{9\baselineskip}
Solving the inequality in Theorem~\ref{thm:characterization}(iv) for
$\nu_p(n)$ gives the following explicit classification.

\begin{corollary}\label{cor:explicit}
A positive integer $n$ is matchable if and only if either it is an
$M$-number, or it has the form $n=p^e m$, where $p$ is prime, $e$ is an integer,
$p\nmid m$, $m$ is an $M$-number, and, with $T:=\tau(m)$,
\[
 T\le p-1,\qquad
 p\le e\le p-1+\left\lfloor\frac{p-1}{T}\right\rfloor.
\]
In the second case $p$ is the unique prime for which $p^p\mid n$,
and $e=\nu_p(n)$ and $m=n/p^e$ are uniquely determined.
\end{corollary}

\begin{proof}
For $e\ge p$, the condition in Theorem~\ref{thm:characterization}(iv)
is equivalent to
\[
 (e+1-p)T<p
 \quad\Longleftrightarrow\quad
 (e+1-p)T\le p-1
 \quad\Longleftrightarrow\quad
 e\le p-1+\left\lfloor\frac{p-1}{T}\right\rfloor.
\]
Since $e+1-p\ge1$, this also forces $T\le p-1$.
\end{proof}

In particular, for every prime $p$ and integer $e\ge0$,
\[
 p^e\text{ is matchable}\quad\Longleftrightarrow\quad e\le2p-2.
\]
For $e=0$, this follows because $1$ is matchable.  For $e\ge1$,
Theorem~\ref{thm:characterization} reduces to
$\lfloor(e+1)/p\rfloor\le1$, which is equivalent to $e+1<2p$.

For example, the only matchable integer divisible by $2^2$ is $4$.
The matchable integers divisible by $3^3$ are $27$, $81$, and $27q$,
where $q\ne3$ is prime.  The family $27q$ was already noted in
\cite{McNewPomerance}.

\Needspace{10\baselineskip}
\begin{proof}[Proof of Corollary~\ref{cor:divisor-closed-intro}]
Let $d\mid n$, where $n$ is matchable.  If $n$ is an $M$-number, then
so is $d$, and Theorem~\ref{thm:strong} applies.
Otherwise write
$n=p^{p+c-1}m$, where $p\nmid m$, $m$ is an $M$-number, and
$c\tau(m)<p$.  Write $d=p^e u$ with $u\mid m$.
If $e\le p-1$, then $d$ is an $M$-number.
If $e\ge p$, put $c':=e+1-p$.  Then $1\le c'\le c$ and
\[
 c'\tau(u)\le c\tau(m)<p.
\]
Theorem~\ref{thm:characterization}(iv) again shows that $d$ is matchable.
\end{proof}

This answers K\"onig's question from the MathOverflow discussion:
every divisor of a matchable integer is matchable.  McNew and Pomerance
observed that every $M$-number divides one of
\[
 M_P:=\prod_{\substack{q\le P\\q\ \mathrm{prime}}}q^{q-1}.
\]
Each $M_P$ is matchable by the Chinese remainder theorem: match
$d\mid M_P$ to the unique $t\in[1,\tau(M_P)]$ satisfying
$t\equiv\nu_q(d)\pmod q$ for every prime $q\le P$.  These assignments
are distinct, and $q\mid d$ implies $q\nmid t$, so the matched pairs
are coprime.  The result about divisors therefore implies that every
$M$-number is matchable.  Our proof that matchability passes to
divisors already uses Theorem~\ref{thm:strong}, which gives strong
matchability directly.

The explicit classification also recovers the density theorem of
McNew and Pomerance~\cite{McNewPomerance}.  Fix a prime $p$ and consider the matchable
integers $n$ for which $p^p\mid n$.  Write $n=p^e m$ with $p\nmid m$.
Corollary~\ref{cor:explicit} bounds both $e$ and $\tau(m)$.  If $\omega(m)$ denotes the number of distinct prime factors
of $m$, then $2^{\omega(m)}\le\tau(m)$, so $\omega(m)$ is bounded.
Integers with a bounded number of distinct prime factors have density
zero.  Since only finitely many exponents of $p$ occur, the
corresponding integers also form a set of density zero.
For every positive integer $X$, the proportion of integers in $[1,X]$
divisible by $p^p$ for some prime $p>P$ is at most
$\sum_{p>P}p^{-p}$.  Combining this bound with the density-zero result
for the finitely many primes $p\le P$, and then letting $P\to\infty$,
shows that the matchable non-$M$-numbers have density zero, in
agreement with \cite{McNewPomerance}.

Once the prime factorization of $n$ is known,
Theorem~\ref{thm:characterization} decides matchability by one
inequality for each prime divisor, without constructing a matching.
The proof also supplies a finite construction of a matching: the
nonzero determinant guarantees a choice of nonzero entries, one in
each row and column, and the progression partition combines the
resulting smaller matchings.  This construction does not provide
a simple closed formula for the required permutation.

The next two sections sharpen the density-zero statement to an
asymptotic formula for the exceptional count.

\section{Counting the matchable non-\texorpdfstring{$M$}{M}-numbers}
\label{sec:exceptional-count}

Recall that $E(x)$ counts the matchable non-$M$-numbers up to $x$.
By Corollary~\ref{cor:explicit}, each integer counted by $E(x)$ has a
unique representation
\[
 n=p^{p+c-1}m,
 \qquad 1\le c<p,\qquad p\nmid m,\qquad c\tau(m)<p,
\]
where $p$ is prime and $m$ is an $M$-number.  For a prime $p$, an integer
$c$ with $1\le c<p$, and a real number $y\ge0$, put
\[
 B_{p,c}(y):=\#\{m\le y:m\text{ is an }M\text{-number},\ 
                     p\nmid m,\ c\tau(m)<p\}.
\]
The uniqueness of the representation gives
\begin{equation}\label{eq:exceptional-sum}
 E(x)=\sum_p\sum_{1\le c<p}
       B_{p,c}\left(\frac{x}{p^{p+c-1}}\right).
\end{equation}
Terms with $p^{p+c-1}>x$ vanish.

Recall that $\omega(m)$ counts the distinct prime factors of $m$,
with $\omega(1)=0$, and let $\mu$ be the M\"obius function.  Thus
$\mu^2(m)=1$ precisely when $m$ is squarefree.  We use the following
consequence of the uniform Sathe--Selberg theorem; see
\cite[Theorem~4]{Selberg} and \cite[Section~II.6.1]{Tenenbaum}.

\begin{proposition}[Sathe--Selberg]\label{prop:sathe-selberg}
Fix $C>0$.  As $y\to\infty$, uniformly for integers
$1\le k\le C\log_3y$, we have
\begin{align}
 \#\{m\le y:\omega(m)=k\}
   &=(1+o(1))\frac{y}{\log y}
                    \frac{(\log_2y)^{k-1}}{(k-1)!},
       \label{eq:sathe-omega}\\
 \#\{m\le y:\mu^2(m)=1,\ \omega(m)=k\}
   &=(1+o(1))\frac{y}{\log y}
                    \frac{(\log_2y)^{k-1}}{(k-1)!}.
       \label{eq:sathe-squarefree}
\end{align}
\end{proposition}

\begin{proof}
For the squarefree estimate, Selberg's theorem applies to the
generating series
\[
 \sum_{m\ge1}\frac{\mu^2(m)z^{\omega(m)}}{m^s}
       =\prod_p\left(1+\frac{z}{p^s}\right),
 \qquad z,s\in\mathbb C,\quad \operatorname{Re}s>1.
\]
The classical uniform formulas contain analytic factors evaluated at
$(k-1)/\log_2y$.  In the stated range this argument tends uniformly
to zero, and both factors tend to $1$.  The relative error terms in
the cited estimates also tend uniformly to zero, which gives the
two assertions.
\end{proof}

For $1\le c<p$, define
\begin{equation}\label{eq:divisor-count-parameters}
 K_{p,c}:=\left\lfloor\frac{p-1}{c}\right\rfloor,
 \qquad
 r_{p,c}:=\left\lfloor\frac{\log K_{p,c}}{\log2}\right\rfloor.
\end{equation}
The condition $c\tau(m)<p$ is equivalent to $\tau(m)\le K_{p,c}$.
Since $\tau(m)\ge2^{\omega(m)}$, it implies
$\omega(m)\le r_{p,c}$.  The next lemma shows that the principal
contribution comes from squarefree integers with exactly $r_{p,c}$
prime factors.

\begin{lemma}\label{lem:bounded-divisor-count}
Fix $C>0$.  As $y\to\infty$, uniformly for primes $p$ and integers
$1\le c<p$ such that $1\le r_{p,c}\le C\log_3y$, we have
\begin{equation}\label{eq:bounded-divisor-count}
 B_{p,c}(y)=(1+o(1))\frac{y}{\log y}
                 \frac{(\log_2y)^{r_{p,c}-1}}{(r_{p,c}-1)!}.
\end{equation}
\end{lemma}

\begin{proof}
Write $r:=r_{p,c}$.  The upper bound follows from
\[
 B_{p,c}(y)\le1+\sum_{1\le k\le r}\#\{m\le y:\omega(m)=k\}
\]
and Proposition~\ref{prop:sathe-selberg}: the term with $k=r$
dominates, because $r/\log_2y=o(1)$.

For the lower bound, retain the squarefree integers with exactly $r$
prime factors.  They are $M$-numbers and satisfy
$\tau(m)=2^r\le K_{p,c}$.  It remains to exclude those divisible by
$p$.  If $r=1$, there is at most one such integer.  Suppose $r\ge2$.
For $p>y^{1/2}$, their number is at most $y/p<y^{1/2}$, which is
negligible compared with the main term in
\eqref{eq:bounded-divisor-count}.  For $p\le y^{1/2}$, write $m=pb$.
Then $b\le y/p$ and $\omega(b)=r-1$.  Since $y/p\ge y^{1/2}$,
we have $r-1=O(\log_3(y/p))$, so
Proposition~\ref{prop:sathe-selberg} applies uniformly at $y/p$.
It bounds their number by
\[
 O\left(\frac{r}{p\log_2y}\right)
       \frac{y}{\log y}\frac{(\log_2y)^{r-1}}{(r-1)!}.
\]
Here $\log(y/p)\ge\tfrac12\log y$ and
$\log_2(y/p)\le\log_2y$, so the estimate is uniform.
Equation~\eqref{eq:sathe-squarefree} now gives the lower bound.
\end{proof}

For $r\ge1$, put
\begin{equation}\label{eq:counting-coefficients}
 a_r:=\sum_{\substack{p\text{ prime},\ c\ge1\\
                       2^rc<p\le2^{r+1}c}}
          \frac1{p^{p+c-1}},
 \qquad
 F(t):=\sum_{r\ge1}a_r\frac{t^{r-1}}{(r-1)!}.
\end{equation}
The bound $a_r\le2\sum_{n>2^r}n^{-n}$ shows that the series defining
$F$ converges for every complex $t$.  These coefficients first give
an asymptotic that includes all choices of $p$ and $c$.

\begin{theorem}\label{thm:count-F}
As $x\to\infty$,
\begin{equation}\label{eq:count-F}
 E(x)\sim\frac{x}{\log x}F(\log_2x).
\end{equation}
\end{theorem}

\begin{proof}
Put $L:=\log_2x$ and $P:=2L/\log L$.  First consider the terms in
\eqref{eq:exceptional-sum} with $p\le P$ and $r_{p,c}\ge1$.
The latter condition gives $c\le(p-1)/2$, and hence, uniformly,
\[
 p^{p+c-1}\le p^{3p/2}\le(\log x)^{3+o(1)}.
\]
Thus $y:=x/p^{p+c-1}$ lies between $x/(\log x)^4$ and $x$ for all
sufficiently large $x$, while
\[
 r_{p,c}\le\frac{\log p}{\log2}=O(\log_3x).
\]
Lemma~\ref{lem:bounded-divisor-count} applies uniformly.  Moreover,
$\log y\sim\log x$ and $\log_2y=L+O(L/\log x)$ uniformly, so
\begin{equation}\label{eq:counting-truncation}
 \sum_{\substack{p\le P,\ 1\le c<p\\r_{p,c}\ge1}}
 B_{p,c}\left(\frac{x}{p^{p+c-1}}\right)
 =(1+o(1))\frac{x}{\log x}
 \sum_{\substack{p\le P,\ 1\le c<p\\r_{p,c}\ge1}}
       \frac{L^{r_{p,c}-1}}{(r_{p,c}-1)!\,p^{p+c-1}}.
\end{equation}
The equality $r_{p,c}=r$ is equivalent to
$2^rc<p\le2^{r+1}c$.  The sum on the right is therefore the
corresponding truncation of $F(L)$.

If $r_{p,c}=0$, then $K_{p,c}=1$, and $m=1$ is the only possible
integer in $B_{p,c}$.  The terms that contribute have
$p\le\log x/\log2$ and $c<p$, so their total contribution is
$O((\log x)^2)$.

For $p>P$, use $B_{p,c}(y)\le y$ and sum the geometric series in $c$:
\[
 \sum_{p>P}\sum_{1\le c<p}
 B_{p,c}\left(\frac{x}{p^{p+c-1}}\right)
 \le2x\sum_{p>P}p^{-p}
 \ll xP^{-P}
 =\frac{x}{(\log x)^{2+o(1)}}
 =o\left(\frac{x}{\log x}\right).
\]

We also need to bound the omitted part of $F(L)$.  Since
$r_{p,c}\le\log p/\log2$, it is at most
\[
 2\sum_{n>P}\exp(g_L(n)),
 \qquad
 g_L(u):=-u\log u+\frac{\log L}{\log2}\log u.
\]
For all sufficiently large $L$ and $u\ge P$,
\[
 g_L'(u)=-\log u-1+\frac{\log L}{u\log2}
            \le-\tfrac12\log P.
\]
The sum is consequently bounded by a geometric series and is
$O(\exp(g_L(P)))$.  Finally,
\[
 g_L(P)=-(2+o(1))L,
\]
so this tail is $o(1)$.  As $F(L)\ge a_1\ge3^{-3}>0$, equation
\eqref{eq:counting-truncation} and the preceding estimates prove
\eqref{eq:count-F}.
\end{proof}

\section{The dominant families and the second term}
\label{sec:dominant-families}

Recall that $q_r$ is the least prime greater than $2^r$.
Bertrand's postulate gives
\begin{equation}\label{eq:least-prime-bounds}
 2^r<q_r<2^{r+1}.
\end{equation}
Thus $(p,c)=(q_r,1)$ contributes to $a_r$.  The next estimate shows
that this single pair accounts for almost all of $a_r$ as $r$ grows.

\begin{lemma}\label{lem:leading-counting-coefficient}
For $r\ge1$,
\begin{equation}\label{eq:leading-counting-coefficient}
 a_r=q_r^{-q_r}\left(1+O\left(\frac1{q_r}\right)\right).
\end{equation}
\end{lemma}

\begin{proof}
Every other pair in \eqref{eq:counting-coefficients} has $p>q_r$:
there is no smaller eligible prime, and $c\ge2$ would require
$p>2^{r+1}$.  For each $p$,
\[
 \sum_{c\ge1}p^{-(p+c-1)}\le2p^{-p}.
\]
For an integer $h\ge1$, the convexity of $u\log u$ gives
\[
 (q_r+h)\log(q_r+h)-q_r\log q_r
       \ge h(\log q_r+1).
\]
Summing the resulting geometric bound yields
\[
 \sum_{n>q_r}n^{-n}\ll\frac{q_r^{-q_r}}{q_r}.
\]
This proves \eqref{eq:leading-counting-coefficient}.
\end{proof}

\Needspace{5\baselineskip}
The coefficient estimate now proves the asymptotic stated in the
introduction.

\begin{proof}[Proof of Theorem~\ref{thm:exceptional-count}]
By Theorem~\ref{thm:count-F}, it is enough to show that
$F(t)\sim\Psi(t)$ as $t\to\infty$.  Fix an integer $R\ge1$.
Lemma~\ref{lem:leading-counting-coefficient} bounds the relative
difference of the terms with $r\ge R$ by $O(1/q_R)$.
The finitely many terms with $r<R$ are negligible compared with
$\Psi(t)$: a later fixed term has a higher power of $t$ and a positive
coefficient.  Hence
\[
 \limsup_{t\to\infty}\left|\frac{F(t)}{\Psi(t)}-1\right|
        \ll\frac1{q_R}.
\]
Letting $R\to\infty$ proves the assertion.
\end{proof}

\begin{remark}\label{rem:dominant-counting-families}
The $r$th term in \eqref{eq:Psi-definition} comes from integers
$q_r^{q_r}m$, where $m$ is squarefree, $\omega(m)=r$, and $q_r\nmid m$.
The squarefree lower bounds in Lemma~\ref{lem:bounded-divisor-count},
summed with the same truncation as in the proof of
Theorem~\ref{thm:count-F}, show that these families together contribute
$(1+o(1))E(x)$.  All other choices of $p$, $c$, or the exponents in $m$
contribute $o(E(x))$.  The first families are
\[
 3^3\ell,\qquad 5^5\ell_1\ell_2,\qquad
 11^{11}\ell_1\ell_2\ell_3,\qquad\ldots,
\]
where the primes in the squarefree part are distinct and differ from
the displayed prime.
\end{remark}

We next estimate the growth of $\Psi$ using only
\eqref{eq:least-prime-bounds}.

\begin{theorem}\label{thm:Psi-growth}
As $t\to\infty$,
\begin{equation}\label{eq:Psi-growth}
 \log\Psi(t)=\frac{\log t}{\log2}
             \bigl(\log_2t-\log_3t+O(1)\bigr).
\end{equation}
Consequently, as $x\to\infty$,
\begin{equation}\label{eq:exceptional-log-growth}
 \log\left(\frac{E(x)\log x}{x}\right)
    =\frac{\log_3x}{\log2}
       \bigl(\log_4x-\log_5x+O(1)\bigr).
\end{equation}
\end{theorem}

\begin{proof}
Put $s:=\log t$.  The logarithm of the $r$th term of $\Psi(t)$ is
\[
 W_r:=(r-1)s-\log((r-1)!)-q_r\log q_r.
\]
For the upper bound, define
\[
 \phi_s(r):=rs-(\log2)r2^r.
\]
Discarding the factorial and using $q_r>2^r$ gives $W_r\le\phi_s(r)$.
To maximize the right-hand side over real $r$, put $u:=r\log2$.
The equation for its critical point is
\[
 e^u(u+1)=\frac{s}{\log2}.
\]
It follows that every integer $r$ maximizing $\phi_s$ satisfies
\begin{equation}\label{eq:dominant-index}
 r=\frac{\log s-\log_2s+O(1)}{\log2},
\end{equation}
and substitution gives
\[
 \max_{r\ge1}\phi_s(r)
     =\frac{s}{\log2}\bigl(\log s-\log_2s+O(1)\bigr).
\]
There are $O(s)$ integers $r\le2s$.  For sufficiently large $s$ and
$r>2s$,
\[
 \phi_s(r)=r\bigl(s-(\log2)2^r\bigr)
              \le-\frac{\log2}{2}r2^r,
\]
so the sum of the remaining terms is bounded.  Summing all terms
therefore adds only $O(\log s)$ to the logarithm of the largest upper
bound.

For the lower bound, choose an integer $r$ satisfying
\eqref{eq:dominant-index}.  Then $2^r\asymp s/\log s$.
By \eqref{eq:least-prime-bounds},
\[
 q_r\log q_r=O(s),\qquad \log((r-1)!)=o(s),
\]
whereas
\[
 rs=\frac{s}{\log2}\bigl(\log s-\log_2s+O(1)\bigr).
\]
Thus this one term has the required logarithm.  This proves
\eqref{eq:Psi-growth}.  Substituting $t=\log_2x$ and using
Theorem~\ref{thm:exceptional-count} gives
\eqref{eq:exceptional-log-growth}.
\end{proof}

The exceptional count is therefore larger than $x/\log x$ by more
than any fixed power of $\log_2x$, while remaining smaller than
$x/(\log x)^{1-\varepsilon}$ for every fixed $\varepsilon>0$.

\begin{corollary}\label{cor:exceptional-count-bounds}
For every fixed $A>0$ and every $\varepsilon>0$,
\[
 \frac{x(\log_2x)^A}{\log x}=o(E(x)),
 \qquad
 E(x)=o\left(\frac{x}{(\log x)^{1-\varepsilon}}\right).
\]
\end{corollary}

\begin{proof}
By \eqref{eq:exceptional-log-growth},
$\log(E(x)\log x/x)$ grows faster than $A\log_3x$ and is
$o(\log_2x)$.  These are the two asserted estimates.
\end{proof}

Finally, recall that $A(x)$ counts all matchable integers up to $x$.  The
exceptional count gives a second term in the density theorem of
McNew and Pomerance~\cite{McNewPomerance}.

\begin{corollary}\label{cor:matchable-second-term}
With $\alpha$ as in Corollary~\ref{cor:strong-density}, we have,
as $x\to\infty$,
\[
 A(x)=\alpha x+\frac{x}{\log x}\Psi(\log_2x)(1+o(1)).
\]
\end{corollary}

\begin{proof}
Let $M(x)$ count the $M$-numbers up to $x$, and let $\pi(z)$ count
the primes up to $z$.  Inclusion--exclusion over primes $p\le z$,
followed by an upper bound for integers divisible by some $p^p$ with
$p>z$, gives
\[
 M(x)=x\prod_{p\le z}\left(1-\frac1{p^p}\right)
       +O\left(2^{\pi(z)}+x\sum_{p>z}\frac1{p^p}\right).
\]
The same tail sum bounds the difference between the finite product
and $\alpha$.  Taking $z:=\log_2x$, and using
$2^{\pi(z)}\le2^z$ and $\sum_{p>z}p^{-p}\ll z^{-z}$, we obtain
\[
 M(x)=\alpha x+o\left(\frac{x}{\log x}\right).
\]
Since $A(x)=M(x)+E(x)$ and $\Psi(t)\to\infty$ as $t\to\infty$,
Theorem~\ref{thm:exceptional-count} completes the proof.
\end{proof}

\section*{Acknowledgments}
We thank Bernardo Recam\'an for posing the original question
and Joachim K\"onig for the observation about divisors in the
MathOverflow discussion \cite{RecamanMO}.  We also thank Carl
Pomerance for inspiring this paper through his work with Nathan
McNew \cite{McNewPomerance} and for connecting our efforts.

\section*{AI Declaration}
BBDR did not use AI in developing his results.  MAA used ChatGPT
(OpenAI; GPT-6 Astra Pro) to assist with literature searches,
developing and checking mathematical arguments, and preparing the
joint manuscript.  The authors take responsibility for the content.

\end{document}